\documentclass[12pt, a4paper]{article}
\usepackage{amsfonts}
\usepackage{mathrsfs}
\usepackage{latexsym}
\usepackage{xy}
\usepackage{amsfonts,amsmath,amssymb,amsthm}
\usepackage{color}

\xyoption{all}

\newcommand{\bcen}{\begin{center}}     \newcommand{\ecen}{\end{center}}
\newcommand{\bay}{\begin{array}}      \newcommand{\eay}{\end{array}}
\newcommand{\beq}{\begin{eqnarray*}}      \newcommand{\eeq}{\end{eqnarray*}}
\def\gl{\mathrm{gl.dim}}
\def\fd{\mathrm{fin.dim}}

\def\dim{\mathrm{dim}}
\def\max{\mathrm{sup}}

\def\Hom{\mathrm{Hom}}
\def\Ext{\mathrm{Ext}}
\def\Tor{\mathrm{Tor}}

\def\mod{\mathrm{mod}}
\def\Mod{\mathrm{Mod}}
\def\id{\mathrm{id}}
\def\pd{\mathrm{pd}}
\def\op{\mathrm{op}}
\def\per{\mathrm{per}}
\def\proj{\mathrm{proj}}

\def\inj{\mathrm{inj}}

\def\RHom{\mathrm{RHom}}

\def\Coker{\mathrm{Coker}}
\def\add{\mathrm{add}}

\begin{document}

\newtheorem{theorem}{Theorem}
\newtheorem{proposition}{Proposition}
\newtheorem{lemma}{Lemma}
\newtheorem{corollary}{Corollary}
\newtheorem{remark}{Remark}
\newtheorem{example}{Example}
\newtheorem{definition}{Definition}
\newtheorem*{conjecture}{Conjecture}
\newtheorem{question}{Question}

\title{\large\bf Recollements and homological dimensions}

\author{\large Yongyun Qin}

\date{\footnotesize College of Mathematics and Information Science,
Qujing Normal University, \\ Qujing, Yunnan 655011, China. E-mail:
qinyongyun2006@126.com}

\maketitle

\begin{abstract} We investigate the behavior of the homological dimensions
under recollements of derived categories of algebras. In particular, we
establish a series of new bounds among the selfinjective dimension or
$\phi$-dimension of the algebras linked by recollements of derived module
categories.
\end{abstract}

\medskip


\medskip


\bigskip

\section{\large Introduction}

Recollements of triangulated categories were introduced by Beilinson Bernstein
and Deligne on perverse sheaves \cite{BBD82}, and they  were extended to algebraic setting
by Cline, Parshall and Scott \cite{CPS88}. Roughly speaking, a recollement
of triangulated categories consists of three such categories
related by sixes functors.
Throughout we focus on recollements of derived categories of algebras, where
all algebras are finite dimensional associative algebras
over an algebraically closed field.
It is known that recollements of derived categories of algebras provide a very useful framework for investigating
the homological connections among these algebras.
For example, Happel studied how the
finiteness of the finitistic dimension of algebras in a recollement interacted on each other \cite{Hap93}, and
some authors discussed the finiteness of the global dimension \cite{Wie91, Koe91, AKLY13}.
Recently, many experts turn to the study of the homological properties or invariants
in the framework of recollements (see \cite{AKLY13, CX12, CX13, Kel98, Pan13, QH16}).
In this paper, we will investigate the behavior of the homological dimensions
under recollements of derived categories of algebras.

   Global dimension and finitistic dimension are classical homological invariants of an algebra,
   and the estimation of these invariants of algebras linked in certain nice
   ways has attracted the interest of many experts. Among others,
   Happel studied the global dimension of two algebras
   involved in a tilting triple, showing that the difference
   between them is less than the projective dimension of the tilting module \cite{Hap88}.
   Later, Keller and Kato generalized this to derived equivalence algebras,
   where the difference is less than the length of the tilting complex (see \cite[Chap.12.5(b)]{GR92}
   and \cite{Kato98}), and then, Pan and Xi showed that this is also true for
   finitistic dimension \cite{PX09}. Moreover, Chen and Xi
   studied the connections between
   the global dimension (resp. finitistic dimension)
   of the algebras involved in a recollement of derived module categories.
  In this paper, we will observe the selfinjective dimension and $\phi$-dimension in recollement situation.

  Let $k$ be an algebraically closed field and $\otimes := -\otimes _k$.
  Assume $A$ is a finite dimensional associative $k$-algebra, and
  $X^\bullet $ is a bounded cochain complex of finitely generated right $A$-modules,
  then there is a unique (up to isomorphism) minimal projective complex $P^\bullet_X$ bounded from right
  such that $X^\bullet\cong P^\bullet_X$ in $D^{b}(\mod A)$. We define $\pd (X_A^\bullet) :=-{\rm min} \{i\ |\ P_X^i\neq0 \}$ and
sup$(X_A^\bullet) :=-{\rm max} \{i\ |\ P_X^i\neq0 \}$.
The {\it homological width}
of $X^\bullet$ is defined as $w(X_A^\bullet):=-\sup(X_A^\bullet)+\pd (X_A^\bullet)$.
  Using this notation, we estimate the right selfinjective dimension of algebras involved in a standard recollement
  (see \cite{Han14} for definition).

\medskip

 {\bf Theorem I.} {\it Let $A$, $B$ and $C$ be algebras, and ($\mathcal{D} B$,\ $\mathcal{D} A$,\ $\mathcal{D} C$,\ $i^*,i_*=i_!,i^!,j_!,j^!=j^*,j_*$)
 a standard recollement defined
 by $X \in D^b(C^{op} \otimes A)$ and $Y \in D^b(A^{op} \otimes B)$. Suppose $Y^*=\RHom_B(Y,B)$.
 Then the following hold true:

{\rm (a)} $\id (C_C)\leq\id (A_A)+w(_CX)$;

{\rm (b)} If $i_*B \in K^{b}(\proj A)$ then $\id (B_B)\leq\id (A_A)+w(Y_B)$;

{\rm (c)} If $i^!A \in K^{b}(\proj B)$ then
$$\id (A_A)\leq \max \{ \id (B_B) +\pd (_AY)+\pd (Y^*_A), \id (C_C)+w(X_A) \}.$$
}

We mention that Theorem 1 extends the derived equivalence case of Kato \cite{Kato98},
and the bounds in Theorem 1 also work for the left selfinjective dimension, global dimension and finitistic dimension.

Next, we turn to the $\phi$-dimension, which is a new homological dimension arising from the
Igusa-Todorov function. In \cite{IT05}, the authors introduced two functions $\phi$ and $\psi$,
sending each finitely generated module to a natural number. These Igusa-Todorov functions
determine new homological measures, generalising the notion of projective dimension, and
have become a powerful tool in the understanding of the finitistic dimension conjecture \cite{IT05, Wei09, Xu13}.
From \cite{HL13, HLM08}, the $\phi$-dimension of an algebra $A$ is
$$\phi \dim(A) :=\sup \{\phi (M) \ | \ M\in \mod A  \}.$$
The $\phi$-dimension of an algebra $A$ has a strong connection with its
global dimension and finitistic dimension: $\fd (A) \leq
\phi \dim(A)\leq \gl(A)$ and they all coincide in the case of $\gl(A)<\infty$.
Moreover, the $\phi$-dimension can be used to describe selfinjective algebras:
an algebra $A$ is selfinjective if and only if $\phi \dim(A)=0$ \cite{HL13}.
Recently, various works were dedicated to study and generalise
the properties of Igusa-Todorov function and the $\phi$-dimension \cite{FLM15, HL13, HLM08, LM16, Xu13}.
In particular, the $\phi$-dimension of an algebra $A$ was characterised in terms of the bi-functors $\Ext_A^i
(-,-)$ and $\Tor_i^A
(-,-)$ in \cite{FLM15}. Using this characterization, the authors prove that the finiteness of the $\phi$-dimension is invariant under
derived equivalence, and they also give a bound for the $\phi$-dimension of two derived equivalence algebras.
In this paper, we will extend this by observing the behavior of the $\phi$-dimensions under
recollements of derived categories of algebras.

\medskip

 {\bf Theorem II.} {\it Let $A$, $B$ and $C$ be finite dimensional algebras, and ($\mathcal{D} B$,\ $\mathcal{D} A$,\ $\mathcal{D} C$,\ $i^*,i_*=i_!,i^!,j_!,j^!=j^*,j_*$)
a standard recollement defined
 by $X \in D^b(C^{op} \otimes A)$ and $Y \in D^b(A^{op} \otimes B)$. Suppose $Y^*=\RHom_B(Y,B)$.
 Then the following hold true:

{\rm (a)} $\phi \dim(C) \leq \phi \dim(A)+w(_CX)$;

{\rm (b)} If $i_*B \in K^{b}(\proj A)$ then $\phi \dim(B) \leq \phi \dim(A)+w(Y_B)$;

{\rm (c)} $\phi \dim(A)\leq \max \{ \phi \dim(B) +\pd (_AY)+\pd (Y^*_A), \phi \dim(C)+w(X_A) \}.$
}

\medskip

The paper is structured as follows. In section 2, we will recall some necessary definitions and
conventions need for the developing of the paper.
Section 3 is about the estimation
of selfinjective dimension, in which Theorem I is obtained.
In section 4 we consider $\phi$-dimension, and we prove Theorem II.

 \section{\large Definitions and conventions}\label{Section-definitions and conventions}

\indent\indent Let $\mathcal{T}_1$, $\mathcal{T}$ and
$\mathcal{T}_2$ be triangulated categories. A {\it recollement} of
$\mathcal{T}$ relative to $\mathcal{T}_1$ and $\mathcal{T}_2$ is
given by
$$\xymatrix@!=4pc{ \mathcal{T}_1 \ar[r]^{i_*=i_!} & \mathcal{T} \ar@<-3ex>[l]_{i^*}
\ar@<+3ex>[l]_{i^!} \ar[r]^{j^!=j^*} & \mathcal{T}_2
\ar@<-3ex>[l]_{j_!} \ar@<+3ex>[l]_{j_*}}$$ and
denoted by ($\mathcal{T}_1$,$\mathcal{T}$,$\mathcal{T}_2$,$i^*,i_*=i_!,i^!,j_!,j^!=j^*,j_*$)
such that

(R1) $(i^*,i_*), (i_!,i^!), (j_!,j^!)$ and $(j^*,j_*)$ are adjoint
pairs of triangle functors;

(R2) $i_*$, $j_!$ and $j_*$ are full embeddings;

(R3) $j^!i_*=0$ (and thus also $i^!j_*=0$ and $i^*j_!=0$);

(R4) for each $X \in \mathcal {T}$, there are triangles

$$\begin{array}{l} j_!j^!X \rightarrow X  \rightarrow i_*i^*X  \rightarrow
\\ i_!i^!X \rightarrow X  \rightarrow j_*j^*X  \rightarrow
\end{array}$$ where the arrows to and from $X$ are the counits and the
units of the adjoint pairs respectively \cite{BBD82}.

In this paper, we focus on recollements of derived categories of algebras.
Let $A$ be a finite dimensional associative
algebra over an algebraically closed field $k$. Denote by $\Mod A$ the
category of right $A$-modules, and by $\mod A$, $\proj A$
and $\inj A$ the full subcategories consisting of all finitely
generated modules, finitely
generated projective modules, and finitely
generated injective modules, respectively.  We denote by $\mathcal{C}(\Mod A)$
the category of cochain complexes over $\Mod A$, with the usual chain map
as morphism.
For
$* \in \{{\rm nothing}, -, +, b \}$, denote by $\mathcal{D}^*(\Mod
A)$ the derived category of complexes over $\Mod
A$ satisfying the corresponding boundedness condition. Denote by
$K^* (\proj A)$ (resp. $K^* (\inj A)$) the homotopy category of of complexes
over $\proj A$ (resp. $\inj A$) satisfying the corresponding boundedness condition,
and $\mathcal{D}^b(\mod A)$ the bounded
derived category of complexes over $\mod A$.
Up to isomorphism, the objects in $K^{b}(\proj A)$ are
precisely all the compact objects in $\mathcal{D}(\Mod A)$. For
convenience, we do not distinguish $K^{b}(\proj A)$ from the {\it
perfect derived category} $\mathcal{D}_{\per}(A)$ of $A$, i.e., the
full triangulated subcategory of $\mathcal{D} A$ consisting of all
compact objects, which will not cause any confusion. Moreover, we
also do not distinguish $K^b(\inj A)$ (resp.
$\mathcal{D}^b(\mod A)$) from their essential images under the
canonical full embeddings into $\mathcal{D}(\Mod A)$.

Usually, we
just write $\mathcal{D} A$ (resp. $\mathcal{C} A$) instead of $ \mathcal{D}(\Mod A)$ (resp. $ \mathcal{C}(\Mod A)$).
Furthermore, for any integers $a, b$ with $a \leq b < \infty$, we denote by $\mathcal{D} (A) _{[a,b]}$
(resp. $\mathcal{C} (A) _{[a,b]}$, $K (\proj A) _{[a,b]}$ and $K (\inj A) _{[a,b]}$) the full subcategory of
$\mathcal{D} A$ (resp. $\mathcal{C} A$, $K (\proj A)$ and $K (\inj A)$) whose objects are $X^\bullet$
with $X^i = 0$ for $i \notin [a, b]$.

For any object $X$ in $\mathcal{D}^-(modA)$, there is a unique (up to isomorphism)
minimal projective complex $P^\bullet_X \in K^-(\proj A)$
such that $X^\bullet\cong P^\bullet_X$ in $\mathcal{D}^{-}(\mod A)$. Here,
a complex is called {\it minimal} if the images
of its differentials lie in the radicals.
Define $\pd (X_A^\bullet) :=-{\rm min} \{i\ |\ P_X^i\neq0 \}$ and
sup$(X_A^\bullet) :=-{\rm max} \{i\ |\ P_X^i\neq0 \}$.
The {\it homological width}
of $X^\bullet$ is defined as $w(X_A^\bullet):=-\sup(X_A^\bullet)+\pd (X_A^\bullet)$. Clearly, sup$(X_A^\bullet)<\infty$
if $X^\bullet \in \mathcal{D}^-(\mod A)$, and both
$w(X_A^\bullet)$ and $\pd (X_A^\bullet)$ are finite if $X^\bullet$ is compact
in $\mathcal{D}A$. In particular, if $X\in \mod A$ then both $w(X)$ and $\pd (X)$ here are equal to
the usual projective dimension of $X$ as a module.

The following two lemmas will be used latter.

\begin{lemma} \label{lemma1}
For any $X^\bullet \in D^-(\mod A)$, $\sup (X_A^\bullet) =-{\rm max} \{i\ |\ H^i(X^\bullet)\neq 0 \}. $
\end{lemma}
\begin{proof}
If $\sup (X_A^\bullet) =-n$, then $P_X^i=0$, $\forall \ i>n$. Thus, $H^i(X^\bullet)=H^i(P_X^\bullet)
=0$ \mbox{for any} $i>n$.
If $H^n(X^\bullet)=0=H^n(P_X^\bullet)$, then the last non-zero differential of
 $P^\bullet_X$ splits, which contradicts to the minimality of $P^\bullet_X$.
 Hence, $H^n(X^\bullet)\neq0$ and $\sup (X_A^\bullet) =-n=-{\rm max} \{i\ |\ H^i(X^\bullet)\neq 0 \}$.
\end{proof}

\begin{lemma} \label{lemma2}
Let $X^\bullet \in D^b(\mod A^{op} \otimes B)$ with $X^\bullet_B\in K^{b}(\proj B)$ and $X^*=\RHom_B(X^\bullet,B)$.
Then $\pd (_BX^*)=-\sup (X_B^\bullet)=-\sup (_AX^\bullet)$ and
$\sup (_BX^*)=-\pd(X_B^\bullet)=\sup (X_A^*)$.
\end{lemma}
\begin{proof}
According to Lemma~\ref{lemma1}, $\sup (X_B^\bullet)=-{\rm max} \{i\ |\ H^i(X^\bullet)\neq 0 \}=\sup (_AX^\bullet)$.
Suppose $\sup (X_B^\bullet)=q$.
Since $X^\bullet_B\in K^{b}(\proj B)$, $\pd (X^\bullet_B)=m<\infty$.
Thus, $X^\bullet_B$ is quasi-isomorphic to a minimal projective complex of the form $$0 \longrightarrow P^{-m} \longrightarrow
P^{-m+1} \longrightarrow \cdots \longrightarrow P^{-q} \longrightarrow 0 .$$
Therefore, $_BX^*$ is quasi-isomorphic to a minimal projective complex $$0 \longrightarrow Q^{q} \longrightarrow
Q^{q+1} \longrightarrow \cdots \longrightarrow Q^{m} \longrightarrow 0 ,$$since
 the functor $\RHom_B(-,B):K^{b}(\proj B)\rightarrow K^{b}(\proj B^{op})$ is an
 equivalence and it sends a minimal complex to a minimal one.
 Hence, $\pd (_BX^*)=-q=-\sup (_AX^\bullet)$
 and $\sup (_BX^*)=-m=-\pd(X_B^\bullet)$. On the other hand, $\sup (X_A^*)=\sup (_BX^*)$ by
 Lemma~\ref{lemma1}.
\end{proof}

Sometimes it is convenient to work in the framework of standard recollement,
where all functors are isomorphic to either derived Hom functor or derived tensor product functor.

\begin{definition}{\rm (\cite[Definition 1]{Han14}) Let $A,B$ and $C$ be algebras.
An recollement ($\mathcal{D} B$,\ $\mathcal{D} A$,\ $\mathcal{D} C$,\ $i^*,i_*=i_!,i^!,j_!,j^!=j^*,j_*$) is
said to be {\it standard} and {\it defined by} $Y \in \mathcal
{D}^b(A^{\op} \otimes B)$ and $X \in \mathcal {D}^b(C^{\op} \otimes A)$
if $i^* \cong -\otimes^L_A Y$ and $j_! \cong -\otimes^L_CX$.
}
\end{definition}

\begin{proposition}\label{lemma-functor} {\rm (\cite[Proposition 2]{Han14})
Let $A,B$ and $C$ be algebras, and ($\mathcal{D} B$,\ $\mathcal{D} A$,\ $\mathcal{D} C$,\ $i^*,i_*=i_!,i^!,j_!,j^!=j^*,j_*$)
a standard recollement defined by $Y \in \mathcal
{D}^b(A^{\op} \otimes B)$ and $X \in \mathcal {D}^b(C^{\op} \otimes A)$. Then
$$\begin{array}{ll}
i^*\cong -\otimes^L_A Y, & j_! \cong -\otimes^L_CX, \\
i_*=\RHom _B(Y, -)=-\otimes^L_B Y^*, & j^!=\RHom _A(X, -)=-\otimes^L_A X^*, \\
i^!=\RHom _A(Y^*, -), & j_*=\RHom _C(X^*, -),\\
\end{array}$$
where $X^*=\RHom_A(X,A)$ and $Y^*=\RHom_B(Y,B)$.
}
\end{proposition}

\begin{proposition}\label{lemma-standard-recollement}{\rm (\cite[Proposition 3]{Han14})
Let $A,B$ and $C$ be algebras. If $\mathcal{D} A$ admit a
recollement relative to $\mathcal{D} B$ and $\mathcal{D} C$, then
$\mathcal{D} A$ admit a standard
recollement relative to $\mathcal{D} B$ and $\mathcal{D} C$.}
\end{proposition}

Owing to Proposition~\ref{lemma-standard-recollement}, we will restrict
our discussion on standard-recollement in the following sections.

\section{\large Recollement and selfinjective dimension}
\indent\indent In this section, we focus on
the selfinjective dimension of algebras involved in a recollement. In particular, using the homological
length and projective dimension of the bimodule complexes in the standard recollement,
we will establish a series of new bounds among the selfinjective dimension of algebras
whose derived categories are involved in a recollement.
In addition, we will show that these bounds also work for
the global dimension and finitistic dimension.

Our first theorem
is a generation of \cite[Proposition 1.7]{Kato98}, which says that
the difference between the
selfselfinjective dimensions of two derived equivalence algebras is less than the term length of
the tilting complex.
The ideal of our proof of Theorem~\ref{theorem1} is similar to that
in \cite{Kato98}, but additional ingredients are needed.

\begin{theorem} \label{theorem1}
Let $A$, $B$ and $C$ be finite dimensional $k$-algebras, and ($\mathcal{D} B$,\ $\mathcal{D} A$,\ $\mathcal{D} C$,\ $i^*,i_*=i_!,i^!,j_!,j^!=j^*,j_*$)
a standard recollement defined
 by $X \in D^b(C^{op} \otimes A)$ and $Y \in D^b(A^{op} \otimes B)$. Suppose $Y^*=\RHom_B(Y,B)$.
 Then the following hold true:

{\rm (a)} $\id (C_C)\leq\id (A_A)+w(_CX)$;

{\rm (b)} If $i_*B \in K^{b}(\proj A)$ then $\id (B_B)\leq\id (A_A)+w(Y_B)$;

{\rm (c)} If $i^!A \in K^{b}(\proj B)$ then
$$\id (A_A)\leq \max \{ \id (B_B) +\pd (_AY)+\pd (Y^*_A), \id (C_C)+w(X_A) \};$$

{\rm (a')} If $j_!DC\in K^{b}(\inj A)$ then $\id (_CC)\leq\id (_AA)+w(_CX)$;

{\rm (b')} If $i_*DB \in K^{b}(\inj A)$ then $\id (_BB)\leq\id (_AA)+w(Y_B)$;

{\rm (c')} If $i^*DA \in K^{b}(\inj B)$ then
$$\id (_AA)\leq \max \{ \id (_BB) +\pd (_AY)+w(Y^*_A), \id (_CC)+w(X_A) \}.$$

\end{theorem}

\begin{proof}

(a) If $w(_CX)=\infty$, (a) holds obviously. If $w(_CX)<\infty$, then
$_cX$ is quasi-isomorphic to a minimal projective complex $P^\bullet$ of the form
$$0 \longrightarrow P^{-\pd (_CX)} \longrightarrow
P^{-\pd (_CX)+1} \longrightarrow \cdots \longrightarrow P^{-\sup (_CX)} \longrightarrow 0 .$$
Thus, for any $M\in \mod C$,
$j_!M\cong M\otimes_C^LX\cong M\otimes_CP^\bullet$ in $\mathcal{D}k$.
Therefore, $H^i(j_!M)=0$ for any $i>-\sup (_CX)$ or $i<-\pd (_CX)$.
As a consequence, $j_!M$ is isomorphic to
a complex
$$0 \longrightarrow X^{-\pd (_CX)} \longrightarrow
\cdots \longrightarrow X^{-\sup (_CX)} \longrightarrow 0 $$
in $\mathcal{D}A$.
On the other hand, $j_!C\cong X_A$ admits a minimal projective resolution $P^\bullet_X$
of the form
$$0 \longrightarrow P^{-\pd (X_A)} \longrightarrow
\cdots \longrightarrow P^{-\sup (X_A)} \longrightarrow 0. $$
Using \cite[Lemma 1.6]{Kato98} or doing induction on the length
of $P^\bullet_X$, we get $\Hom _{\mathcal{D}(\Mod A)}(j_!M,j_!C[i])=0$,
for any $i>\id A_A-\sup (X_A)+\pd (_CX)=\id A_A-\sup (_CX)+\pd (_CX)=\id A_A+w(_CX)$.
Therefore, for any $M\in \mod C$, $$\Ext _C^i(M,C)=\Hom  _{\mathcal{D}(\Mod C)}(M,C[i])
=\Hom  _{\mathcal{D}(\Mod A)}(j_!M,j_!C[i])=0,$$
for any $i>\id A_A+w(_CX)$. That is, $\id C_C\leq\id A_A+w(_CX)$.

\medskip

(b) Since $Y_B \cong i^*A \in K^{b}(\proj B)$, it follows from Lemma~\ref{lemma2}
that $\pd (_BY^*)=-\sup (Y_B)$ and
$\sup (_BY^*)=-\pd(Y_B)$. Hence,
$_BY^*$ is quasi-isomorphic to
a projective complex $Q^\bullet \in K(\proj B^{\op})_{[\sup (Y_B),\pd (Y_B) ]}$.
Thus, for any $M\in \mod B$, $i_*M\cong M\otimes_B^LY^* \in\mathcal{D}A_{[\sup (Y_B), \pd (Y_B)]}$,
up to quasi-isomorphic.
On the other hand, $i_*B=Y^*_A \in K(\proj A)_{[-\pd (Y_A^*), \pd (Y_B)]}$.
If $i_*B \in K^{b}(\proj A)$, then $\pd (Y_A^*)<\infty$ and thus by \cite[Lemma 1.6]{Kato98},
$\Hom _{\mathcal{D}(\Mod A)}(i_*M,i_*B[i])=0$,
for any $i>\id A_A+\pd (Y_B)-\sup (Y_B)=\id A_A+w(Y_B)$.
Therefore,$$\Ext _B^i(M,B)=\Hom  _{\mathcal{D}(\Mod B)}(M,B[i])
=\Hom  _{\mathcal{D}(\Mod A)}(i_*M,i_*B[i])=0,$$
for any $i>\id A_A+w(Y_B)$ and $M\in \mod B$. That is, $\id B_B\leq\id A_A+w(Y_B)$.

\medskip

(c) If one of $\pd (_AY)$, $\pd (Y^*_A)$ and $w(X_A)$ is infinite, (c) holds true.
Now assume all of $\pd (_AY)$, $\pd (Y^*_A)$ and $w(X_A)$ are finite.

For any $M\in \mod A$, there is a triangle
$$j_!j^*M \rightarrow M \rightarrow i_*i^*M \rightarrow $$ in $\mathcal{D}(\Mod A)$.
For any $i \in \mathbb{Z}$,
applying the functor $\Hom_{\mathcal{D}(\Mod A)}(-, A[i])$ to
this triangle, we obtain exact sequence
$$\Hom_{\mathcal{D}(\Mod A)}(i_*i^*M, A[i])\rightarrow \Hom_{\mathcal{D}(\Mod A)}
(M, A[i])\rightarrow \Hom_{\mathcal{D}(\Mod A)}(j_!j^*M, A[i]).$$

Now we prove (c) by two steps.

\medskip

{\it Step 1.} We claim $$\Hom_{\mathcal{D}(\Mod A)}(i_*i^*M, A[i])=0, \ \forall\ i> \id (B_B) +\pd (_AY)+\pd (Y^*_A).$$

\medskip

Indeed, $\Hom_{\mathcal{D}(\Mod A)}(i_*i^*M, A[i])\cong \Hom_{\mathcal{D}(\Mod B)}(i^*M, i^!A[i])$,
and clearly, $i^*M\cong M\otimes_A^LY\in\mathcal{D}B_{[-\pd (_AY),-\sup (_AY)]}$,
up to quasi-isomorphic.
By Proposition~\ref{lemma-functor} and
Lemma~\ref{lemma2}, we have
$\sup (i^!A)=\sup(\RHom_A(Y^*, A))=-\pd(Y_A^*)$.
Since $i^!A\in K^{b}(\proj B)$, it follows that $i^!A\in K(\proj B)_{[t, \pd(Y_A^*)]}$
for some $t\in \mathbb{N}$.
Hence, by \cite[Lemma 1.6]{Kato98},
we have $$\Hom_{\mathcal{D}(\Mod B)}(i^*M, i^!A[i])=0, \ \forall\ i> \id (B_B) +\pd (_AY)+\pd (Y^*_A).$$

\medskip

{\it Step 2.} We claim $$\Hom_{\mathcal{D}(\Mod A)}(j_!j^*M, A[i])=0, \ \forall\ i> \id (C_C)+w(X_A).$$

\medskip

Indeed, $\Hom_{\mathcal{D}(\Mod A)}(j_!j^*M, A[i])\cong \Hom_{\mathcal{D}(\Mod B)}(j^*M, j^*A[i])$.
By Lemma~\ref{lemma2}, we have $\pd (_AX^*)=-\sup (X_A^\bullet)$ and
$\sup (_AX^*)=-\pd(X_A^\bullet)$. Thus,
$_AX^*$ is quasi-isomorphic to
a projective complex $P^\bullet \in K(\proj A^{\op})_{[\sup (X_A),\pd (X_A)]}$.
Therefore, $j^*M\cong M\otimes_A^LX^*\in\mathcal{D}C_{[\sup (X_A),\pd (X_A)]}$,
up to quasi-isomorphic.
On the other hand, $\sup(j^*A)=\sup(X_C^*)=\sup (_AX^*)=-\pd(X_A^\bullet)$,
and our assumption $\pd (Y^*_A)<\infty$ implies
$j^*A\in K^{b}(\proj C)$ (see \cite[Lemma 4.3]{AKLY13}). Hence, $j^*A$ is quasi-isomorphic to
a projective complex $Q^\bullet \in K(\proj C)_{[u,\pd (X_A)]}$, for some $u\in\mathbb{N}$.
Therefore, by \cite[Lemma 1.6]{Kato98},
we have $\Hom_{\mathcal{D}(\Mod B)}(j^*M, j^*A[i])
\linebreak =0$, for any $i> \id (C_C)+\pd (X_A)-\sup (X_A)
=\id (C_C)+w(X_A).$

\medskip

Combining the above two claims and the exact sequence
$$\Hom_{\mathcal{D}(\Mod A)}(i_*i^*M, A[i])\rightarrow \Hom_{\mathcal{D}(\Mod A)}
(M, A[i])\rightarrow \Hom_{\mathcal{D}(\Mod A)}(j_!j^*M, A[i]),$$
we get $\Ext _A^i(M,A)=\Hom_{\mathcal{D}(\Mod A)}
(M, A[i])=0$, for any $i> \max \{ \id (B_B) +\pd (_AY)+\pd (Y^*_A), \id (C_C)+w(X_A) \}.$
Therefore, $$\id (A_A)\leq \max \{ \id (B_B) +\pd (_AY)+\pd (Y^*_A), \id (C_C)+w(X_A) \}.$$

\medskip

(a'), (b') and (c') can be proved similarly, just need to consider
the projective dimension of $DC$ (resp. $DB$ and $DA$) as a right $C$ (resp. $B$ and $A$)-module.
Here we only write down the proof of (a').

If $w(_CX)=\infty$, then we are done. Assume $\pd (_CX)<\infty$.
Then $_CX$ is quasi-isomorphic to
a projective complex $P^\bullet \in K(\proj C^{\op})_{[-\pd (_CX),-\sup (_CX)]}$.
Hence, for any $ M\in \mod C$, $j_!M\cong M\otimes_C^LX\in\mathcal{D}A_{[-\pd (_CX),-\sup (_CX)]}$,
up to quasi-isomorphic.
Since $j_!DC\in K^{b}(\inj A)$,
$j_!DC$ is quasi-isomorphic to
a injective complex $I^\bullet \in K(\inj A)_{[-\pd (_CX),t]}$,
for some $t\in \mathbb{N}$.
Therefore, by \cite[Lemma 1.6]{Kato98},
we have $$\Ext _C^i(DC,M)=\Hom _{\mathcal{D}(\Mod C)}(DC,M[i])=\Hom _{\mathcal{D}(\Mod A)}(j_!DC,j_!M[i])=0,$$
for any $i>\pd(DA_A)-\sup(_CX)+\pd(_CX) =\pd(DA_A)+w(_CX)$.
Hence, $\pd(DC_C)\leq\pd(DA_A)+w(_CX)$, that is, $\id _CC\leq\id _AA+w(_CX)$.

\end{proof}

As a consequence of Theorem~\ref{theorem1}, we reobtain the following result
of Kato.
\begin{corollary}{\rm (\cite[Proposition 1.7]{Kato98})}
Let $A$ and $B$ be two derived equivalence algebras and $_BP_A^\bullet$
the corresponding two-side tilting complex. Then

{\rm (1)} $\id (A_A)-w(P_A) \leq \id (B_B)\leq\id (A_A)+w(P_A) ;$

{\rm (2)} $\id (_AA)-w(P_A) \leq \id (_BB)\leq\id (_AA)+w(P_A) .$
\end{corollary}

\begin{proof}
Applying Theorem~\ref{theorem1} to the trivial recollement
$$\xymatrix@!=4pc{ 0 \ar[r] & \mathcal{D}A \ar@<-3ex>[l]
\ar@<+3ex>[l] \ar[r]^{-\otimes_A^LP_B^*} & \mathcal{D}B
\ar@<-3ex>[l]_{-\otimes_B^LP_A} \ar@<+3ex>[l]}$$
and noting that $w(_BP)=w(P_A)$, we obtain the desired statement.
\end{proof}

\begin{corollary}\label{corollary-triangular-matrix-injdim}
Let $B$ and $C$ be finite dimensional algebras, $M$ a finitely
generated $C$-$B$-bimodule, and $A = \left[\begin{array}{cc} B & 0
\\ M & C \end{array}\right] $. Then, the following statements hold.

{\rm (i)} $\id (B_B)\leq \id (A_A)$ and $\id (_CC)\leq \id (_AA)$.

{\rm (ii)} If $\pd (M_B)<\infty$, then $\id (A_A)\leq \sup\{\id (B_B)+\pd(_CM)+1, \id(C_C) \}$.

{\rm (iii)} If $\pd (_CM)<\infty$, then $\id (_AA)\leq \sup\{\id (_CC)+\pd(M_B)+1, \id(_BB) \}$.

\end{corollary}

\begin{proof}
Let $e_1=\left[\begin{array}{cc} 1 & 0
\\ 0& 0 \end{array}\right]$ and $e_2=\left[\begin{array}{cc} 0 & 0
\\ 0& 1 \end{array}\right]$. By \cite[Example 3.4]{AKLY13}, there is
a $2$-recollement $$\xymatrix@!=8pc{ \mathcal{D}B \ar@<+1.5ex>[r]|{-\otimes ^L_Be_1A}
\ar@<-4.5ex>[r] & \mathcal{D}A \ar@<+1.5ex>[r]|{-\otimes_A^LAe_2} \ar@<-4.5ex>[r]
\ar@<-4.5ex>[l]|{-\otimes_A^L A/Ae_2A} \ar@<+1.5ex>[l]|{-\otimes ^L_AAe_1} &
\mathcal{D}C\ar@<-4.5ex>[l]|{-\otimes ^L_Ce_2A}
\ar@<+1.5ex>[l]|{-\otimes _C^L A/Ae_1A}}.$$
Clearly, $\pd ((Ae_1)_B)=\pd(M_B)$, $\pd (_A(A/Ae_2A))=\pd(_CM)+1$,
$\pd ((A/Ae_1A)_A)
\linebreak =\pd(M_B)+1$, and the condition
$DA\otimes_A^LAe_2\in K^{b}(\inj C)$ is equivalent to
$\pd (_CM)<\infty$. Now the corollary follows from Theorem~\ref{theorem1}.
\end{proof}
Recall that an algebra $A$ is called {\it Gorenstein} if $\id (_AA)<\infty$
and $\id (A_A)<\infty$. In this case, $\id (_AA)<n$ if and only if
$\id (A_A)<n$ (\cite[Lemma A]{Zak69}). An Gorenstein algebra $A$ is {\it $n$-Gorenstein} if
$\id (_AA)<n$.

For a finite dimensional algebra $A$ and $n\in \mathbb{N}^+$, denote
$$T_n(A) = \left(\begin{array}{ccccc} A & 0 & \cdots &0 &0
\\  A & A & \cdots &0 &0 \\  \vdots & \vdots & \ddots &\vdots &\vdots
\\ A & A & \cdots &A &0 \\ A & A & \cdots &A &A\end{array}\right).$$

\begin{corollary}{\rm (\cite[Lemma 4.1]{XZ12})} $T_n(A)$ is $(m + 1)$-Gorenstein
if and only if $A$ is $m$-Gorenstein.
\end{corollary}
\begin{proof}
Set $M=\left(\begin{array}{c} A \\ A \\ \vdots \\ A\end{array}\right)_{(n-1)\times 1}$.
Regarding $M$ as a natural $T_{n-1}(A)$-$A$-bimodule via the
multiplication of matrices, then $T_n(A) = \left[\begin{array}{cc} A & 0
\\ M & T_{n-1}(A) \end{array}\right] $.
Obviously, both $_{T_n(A)}M$ and $M_A$ are projective modules. Therefore,
$T_n(A)$ is Gorenstein
if and only if $A$ is Gorenstein (Ref. \cite[Theorem 3.3]{Chen09} or \cite[Theorem 2.2]{XZ12}).

Assume $A$ is $m$-Gorenstein. By Corollary~\ref{corollary-triangular-matrix-injdim},
$\id(T_n(A)_{T_n(A)})\leq \sup \{m+1, \id(T_{n-1}(A)_{T_{n-1}(A)})\}$,
and thus, we get $\id(T_n(A))\leq m+1$ by induction.

Conversely, if $\id(T_n(A))_{T_n(A)})\leq m+1$, then by Corollary~\ref{corollary-triangular-matrix-injdim},
$\id (A_A)\leq m+1$ and
$\id (T_2(A)_{T_2(A)})\leq m+1$ .
Assume there exist an module $X\in \mod A$ with $\id (X_A)=m+1$.
Then the injective dimension of $(X_A,0)$ as a right $T_2(A)$ module is
$m+2$, which conflicts to $\id (T_2(A)_{T_2(A)})\leq m+1$.
Therefore, $\id (A_A)\leq m$.

\end{proof}

Next, we mention that our technique in Theorem~\ref{theorem1} can also be applied
to global dimension and finitistic dimension.  Here we just write down these theorems
without proof.

\begin{theorem}{\rm (See \cite[Theorem 3.18]{CX13})}\label{thoerem-global-dim}
Let $A$, $B$ and $C$ be finite dimensional $k$-algebras, and $\mathcal{D} A$ admit a standard
recollement relative to $\mathcal{D} B$ and $\mathcal{D} C$ defined by
$X \in D(\Mod C^{op} \otimes A)$ and $Y \in D(\Mod A^{op} \otimes B)$. Set $Y^*:=\RHom_B(Y,B)$.
Then
  $$\max \{{\rm gl.dim}B-w(Y_B),{\rm gl.dim}C-w(_CX)\} \leq {\rm gl.dim}A $$
 $$ \leq   \max \{{\rm gl.dim}B + \pd(_AY)+\pd (Y_A^*), {\rm gl.dim}C + w(X_A) \}. $$
\end{theorem}

\begin{theorem}{\rm (See \cite[Theorem 3.11]{CX13})}
Let $A$, $B$ and $C$ be finite dimensional $k$-algebras, and ($\mathcal{D} B$,\ $\mathcal{D} A$,\ $\mathcal{D} C$,\ $i^*,i_*=i_!,i^!,j_!,j^!=j^*,j_*$)
a standard recollement defined by
$X \in D(\Mod C^{op} \otimes A)$ and $Y \in D(\Mod A^{op} \otimes B)$. Set $Y^*:=\RHom_B(Y,B)$.
Then

{\rm (a)} $\fd C\leq\fd A+w(_CX)$;

{\rm (b)} If $i_*B \in K^{b}(\proj A)$ then $\fd B\leq\fd A+w(Y_B)$;

{\rm (c)} $\fd (A_A)\leq \max \{ \fd (B_B) +\pd (_AY)+\pd (Y^*_A), \fd (C_C)+w(X_A) \}.$
\end{theorem}

\subsection{Recollement and $\phi$-dimension}
\indent\indent In this section, we will observe the behavior of the
$\phi$-dimension under recollements of derived categories of algebras.

We start now by recalling the definition of the Igusa-Todorov function $\phi$ :
Obj$(\mod A) \rightarrow \mathbb{N}$.
For a finite dimensional algebra $A$, let $K(A)$ denote the quotient of the free abelian group generated by
the set of iso-classes $\{[M] : M \in \mod A\}$ modulo the relations: (a) $[N] - [S] - [T]$ if
$N \cong S \oplus T $ and (b) $[P]$ if $P$ is projective.
Then $K$ is the free abelian group generated by all isomorphism classes of finitely
generated indecomposable non projective $A$-modules.
The syzygy functor $\Omega$ gives rise to a group homomorphism $\Omega$ : $K\rightarrow K$. For any $M \in \mod A$, let
$\langle M \rangle$ denote the subgroup of $K$ generated by all the indecomposable non projective
summands of $M$. Since the rank of
$\Omega (\langle M \rangle)$ is less or equal to the rank of $\langle M \rangle$ which
is finite, it follows from the well ordering principle that there exists a non-negative
integer $n$ such that the rank of
$\Omega ^n (\langle M \rangle)$ is equal to the rank of
$\Omega ^i (\langle M \rangle)$ for all
$i \geq n$. We let $\phi (M)$ denote the least such $n$.

For any $M \in \mod A$, $\phi (M)$ is a good refinement of its projective
dimension $\pd (M)$. Indeed, $\phi (M)=\pd (M)$ if if $\pd (M) < \infty$;
and $\phi (M)$ is always finite even if $\pd (M) = \infty$ (Ref. \cite{HLM08,IT05}).

Now we introduce the natural concept of $\phi$-dimension of a finite dimensional algebra $A$.

\begin{definition}
The $\phi$-dimension of an algebra $A$ is defined as follows $$\phi \dim(A) = \sup \{ \ \phi(M)\ | \ M \in \mod A \}.$$
\end{definition}

In \cite{FLM15}, the authors give another description of the $\phi$-dimension
in terms of the bi-functor $\Ext_A^i
(-,-)$. Let's explain their work in detail.

\begin{definition}
Let $d$ be a positive integer, $M \in \mod A$, and
$ X, Y \in\add (M)$ with $\add (X) \cap \add (Y)=0$. The pair
$(X, Y)$ is called a $d$-Division of $M$ if
$\Ext ^d_A(X,-) \ncong \Ext ^d_A(Y,-)$
and $\Ext ^{d+1}_A(X,-) \cong \Ext ^{d+1}_A(Y,-)$.
\end{definition}

\begin{theorem}{\rm (\cite[Theorem 3.6]{FLM15})}\label{theorem-phi}
Let $A$ be a finite dimensional algebra and $M \in \mod A$. Then
$\phi (M)= \max (\{d \in \mathbb{N} : \mbox{there is a d-Division of M}\} \cup {0})$.
\end{theorem}

With this homological description, S. Fernandes, M. Lanzilotta and O. Mendoza
showed that the difference between the $\phi$-dimension of two derived equivalence
algebras is less than the length of the tilting complex. Now we will extend this
in the framework of recollement. The following lemma will be used latter.

\begin{lemma}{\rm (\cite[Corollary 3.3]{FLM15})}\label{corollary-ext}
Let $A$ be a finite dimensional algebra and $M,N \in \mod A$. Then, the following
conditions are equivalent.

{\rm (a)} $\Ext ^{n+1}_A(X,-) \cong \Ext ^{n+1}_A(Y,-)$;

{\rm (b)} $\Ext_C^t(M,-)\cong \Ext_C^t(N,-)$ for any $t\geq n+1$.
\end{lemma}

\begin{lemma}\label{lemma-projective-truncation}
Let $A$ be a finite dimensional algebra and $X^\bullet \in \mathcal{C} (A) _{[a,b]}$.
Then $X^\bullet$ is quasi-isomorphism to complexes $U^\bullet _X$ and $V^\bullet _X$, where

{\rm (a)} $U^\bullet _X\in \mathcal{C} (A) _{[a,b]}$ with $U^i_X\in \proj A,\ \forall \ i\in [a+1,b];$

{\rm (b)} $V^\bullet _X\in \mathcal{C} (A) _{[a,b]}$ with $V^i_X\in \inj A,\ \forall \ i\in [a,b-1].$
\end{lemma}

\begin{proof}
(a) Since $X^\bullet \in \mathcal{C} (A) _{[a,b]}$, $X$ is quasi-isomorphic to
a projective complex $(P^\bullet, d)$ of the form
$$\cdots \longrightarrow P^{a} \longrightarrow P^{a+1}\longrightarrow
\cdots \longrightarrow P^{b} \longrightarrow 0 .$$
Clearly, the projection $\pi :\ P^{a}\longrightarrow \Coker d^{a-1}$
induce a quasi-isomorphic between the complex $P^\bullet $ and
$$0 \longrightarrow \Coker d^{a-1} \longrightarrow P^{a+1}\longrightarrow
\cdots \longrightarrow P^{b} \longrightarrow 0 ,$$
and thus we obtain the desired complex.

(b) This can be proven in a similar way as we did in (a).
\end{proof}

\begin{theorem}\label{theorem4}
Let $A$, $B$ and $C$ be finite dimensional algebras, and ($\mathcal{D} B$,\ $\mathcal{D} A$,\ $\mathcal{D} C$,\ $i^*,i_*=i_!,i^!,j_!,j^!=j^*,j_*$)
a standard recollement defined
 by $X \in D^b(C^{op} \otimes A)$ and $Y \in D^b(A^{op} \otimes B)$. Suppose $Y^*=\RHom_B(Y,B)$.
 Then the following hold true:

{\rm (a)} $\phi  \dim(C) \leq \phi  \dim(A)+w(_CX)$;

{\rm (b)} If $i_*B \in K^{b}(\proj A)$ then $\phi  \dim(B) \leq \phi  \dim(A)+w(Y_B)$;

{\rm (c)} $\phi  \dim(A)\leq \max \{ \phi  \dim(B) +\pd (_AY)+\pd (Y^*_A), \phi  \dim(C)+w(X_A) \}.$

\end{theorem}

\begin{proof}
(a) If $w(_CX)=\infty$, (a) holds obviously. Assume that $w(_CX)=n <\infty$.
If $\phi  \dim(C) \leq n$ then $\phi  \dim(C) \leq \phi  \dim(A)+n$. If not,
there exists $L\in \mod C$
such that $\phi _C(L)=d>n$.
It follows from Theorem~\ref{theorem-phi} and Lemma~\ref{corollary-ext} that there exist $M, N\in \add L$
such that
$$\left\{\begin{array}{ll}\Ext_C^d(M,-)\ncong \Ext_C^d(N,-),& (1)
\vspace{2mm}  \\ \Ext_C^t(M,-)\cong \Ext_C^t(N,-) \  \mbox{for}\   t\geq d+1.\ \ &(2)\end{array}\right.$$
Let's prove (a) by the following three steps.

\medskip

{\it Step 1.} We claim $$\Hom_{\mathcal{D}C}(M,j^*(-)[t]) \cong\Hom_{\mathcal{D}C}(M,j^*(-)[t]) | _{\mod A},
\forall \ t\geq d+ \pd (X_A)+1,$$ and thus by adjointness, we have
$$\Hom_{\mathcal{D}C}(j_!M,-[t]) \cong\Hom_{\mathcal{D}C}(j_!N,-[t]) | _{\mod A},
\forall \ t\geq d+ \pd (X_A)+1.$$

\medskip

Indeed, for any $S\in \mod A$, $j^*S\cong S\otimes_A^LX^* \in \mathcal{D}(C)_{[\sup (X_A),
\pd (X_A)]}$ by the proof of Theorem~\ref{theorem1}.
For the sake of simplicity, we define $a:=\sup (X_A)$, $b:=\pd (X_A)$,
and sometimes we denote the bi-functor $\Hom _{\mathcal{D}C}(-,-)$ by
$(-,-)$.
It follows from Lemma~\ref{lemma-projective-truncation} that $j^*S $ is quasi-isomorphic to
a complex $V^\bullet \in \mathcal{D} (C) _{[a,
b]}$ with $V^i\in \inj A$, for all $i\in [a,
b-1].$
Applying the functor $\Hom_{\mathcal{D}C}(M,-[t])$ to the triangle
$$V^{b}[-b] \longrightarrow V^\bullet \longrightarrow \tau _{\leq b-1}V^\bullet \longrightarrow $$
(where $\tau$ is the brutal truncation, that is, $\tau _{\leq b-1}V^\bullet$
is the complex
\linebreak
$0 \longrightarrow V^{a} \longrightarrow
\cdots \longrightarrow V^{b-1} \longrightarrow 0 )$,
we get an exact sequence
$$(M,\tau _{\leq b-1}V^\bullet[-1][t])\longrightarrow
(M,V^{b}[-b][t])\longrightarrow
(M,V^\bullet[t])\longrightarrow(M,\tau _{\leq b-1}V^\bullet[t]).$$
By \cite[Lemma 1.6]{Kato98}, we have $(M,\tau _{\leq b-1}V^\bullet[t])=0$,
for any $t> b-1$. Therefore, $(M,V^{b}[-b][t])\cong
(M,V^\bullet[t])$, for any $t> b$. And similarly, we obtain that \linebreak $(N,V^{b}[-b][t])\cong
(N,V^\bullet[t])$, for any $t> b$. On the other hand,
(2) implies that $(M,V^{b}[-b][t])\cong (N,V^{b}[-b][t])$, for any $t-b\geq d+1$,
and thus, $(M,V^\bullet[t])\cong (N,V^\bullet[t])$, for any $t\geq d+b+1$.
Therefore, $(M,j^*S [t])\cong (N,j^*S [t])$, for any $t\geq d+b+1$ and $S\in \mod C$.

\medskip

{\it Step 2.} Since $w(_CX)<\infty$, we get $\pd(_CX)<\infty$. Then,
$j_!M\cong M\otimes_C^LX
\linebreak \in \mathcal{D}(A)_{[-\pd (_CX),
-\sup (_CX)]}$ and $j_!N\in \mathcal{D}(A)_{[-\pd (_CX),
-\sup (_CX)]}$. Let $\pd(_CX)=c$. Since $\sup (_CX)=\sup (X_A)=a$, we obtain that
$j_!M$, $j_!N\in \mathcal{D}(A)_{[-c,-a]}$.
By Lemma~\ref{lemma-projective-truncation}, the complexes $j_!M $ and $j_!N $
can be replaced, respectively, by $Z^\bullet , W^\bullet \in \mathcal{D} (A) _{[-c,
-a]}$ with $Z^i, W^i\in \proj A,\ \forall \ i\in [-c+1,
-a].$ Now we claim that $\Ext_A^t(Z^{-c},-)\cong\Ext_A^t(W^{-c},-)$, for any $t\geq d-c+b+1$.

\medskip

Indeed, for any $S\in \mod A$, applying the functor $\Hom_{\mathcal{D}A}( -,S[t])$ to the triangle
$$\tau_{\geq -c+1}Z^\bullet \longrightarrow Z^\bullet \longrightarrow Z^{-c}[c]\longrightarrow ,$$
we get an exact sequence
$$(\tau_{\geq -c+1}Z^\bullet,S[t-1])\longrightarrow
(Z^{-c}[c],S[t])\longrightarrow
(Z^\bullet,S[t])\longrightarrow(\tau_{\geq -c+1}Z^\bullet,S[t]).$$
By \cite[Lemma 1.6]{Kato98}, we have $(\tau_{\geq -c+1}Z^\bullet,S[t])=0$,
for any $t> c-1$. Therefore, $(Z^{-c}[c],S[t])\cong
(Z^\bullet,S[t])$, for any $t> c$. And similarly, we
obtain that\linebreak $(W^{-c}[c],S[t])\cong
(W^\bullet,S[t])$, for any $t> c$. On the other hand,
step one implies that $(Z^\bullet,S[t])\cong (W^\bullet,S[t])$, for any $t\geq d+b+1$,
and thus, $(Z^{-c}[c],S[t])\cong (W^{-c}[c],S[t])$, for any $t\geq \max \{ d+b+1, c+1\}$.
Therefore, $(Z^{-c},S[t])\cong (W^{-c},S[t])$, for any $t\geq \max \{ d+b-c+1, 1 \}$.
Since $d>n=-a+c$, $d+b-c+1>-a+b+1>1$, and then $(Z^{-c},S[t])\cong (W^{-c},S[t])$, for any $t\geq d+b-c+1$.

\medskip

{\it Step 3.} We claim $\Ext_A^{d+a-c}(Z^{-c},-)\ncong\Ext_A^{d+a-c}(W^{-c},-)$.

\medskip

Assume $\Ext_A^{d+a-c}(Z^{-c},-)\cong\Ext_A^{d+a-c}(W^{-c},-)$. Then
$(Z^{-c},j_!T[d-c])\cong (W^{-c},j_!T[d-c])$,
for any
$T \in \mod C$.
Indeed, since $j_!T\cong T\otimes ^L_CX\in \mathcal{D}(A)_{[-c,-a]}$,
it follows from Lemma~\ref{lemma-projective-truncation} that
$j_!T$ is quasi-isomorphism to a complex $J^\bullet \in \mathcal{D}(A)_{[-c,-a]}$ with $J^i\in \inj A,\ \forall \ i\in [-c,
-a-1].$
Applying the triangle functor $\Hom_{\mathcal{D}A}(Z^{-c},-[d-c])$ to the triangle
$$J^{-a}[a] \longrightarrow J^\bullet \longrightarrow \tau _{\leq -a-1}J^\bullet \longrightarrow ,$$
we get an exact sequence
$$(Z^{-c},\tau _{\leq -a-1}J^\bullet[d-c-1])\longrightarrow
(Z^{-c},J^{b}[d-c+a])$$ $$\longrightarrow
(Z^{-c},J^\bullet[d-c])\longrightarrow(Z^{-c},\tau _{\leq -a-1}J^\bullet[d-c]).$$
Since $d-c>-a+c-c=-a$, it follows from \cite[Lemma 1.6]{Kato98} that
$(Z^{-c},\tau _{\leq -a-1}J^\bullet[d-c])=0=(Z^{-c},\tau _{\leq -a-1}J^\bullet[d-c-1])$. Therefore,
\linebreak
$(Z^{-c},J^{b}[d-c+a]) \cong
(Z^{-c},J^\bullet[d-c])$. And similarly, we
obtain that
\linebreak
 $(W^{-c},J^{b}[d-c+a]) \cong
(W^{-c},J^\bullet[d-c])$. Therefore, $(Z^{-c},J^\bullet[d-c])\cong (W^{-c},J^\bullet[d-c])$
and thus $(Z^{-c},j_!T[d-c])\cong (W^{-c},j_!T[d-c])$, for any
$T \in \mod C$.

Now we claim $(Z^\bullet,j_!T[d])\cong (W^\bullet,j_!T[d])$, for any $T\in \mod A$. Indeed,
applying the functor $\Hom_{\mathcal{D}A}( -,j_!T[d])$ to the triangle
$$\tau_{\geq -c+1}Z^\bullet \longrightarrow Z^\bullet \longrightarrow Z^{-c}[c]\longrightarrow ,$$
we get an exact sequence
$$(\tau_{\geq -c+1}Z^\bullet,j_!T[d-1])\longrightarrow
(Z^{-c}[c],j_!T[d])\longrightarrow
(Z^\bullet,j_!T[d])\longrightarrow(\tau_{\geq -c+1}Z^\bullet,j_!T[d]).$$
Since $d>-a+c$, it follows from \cite[Lemma 1.6]{Kato98} that $(\tau_{\geq -c+1}Z^\bullet,j_!T[d])=0
=(\tau_{\geq -c+1}Z^\bullet,j_!T[d-1])$. Therefore, $(Z^{-c}[c],j_!T[d])\cong
(Z^\bullet,j_!T[d])$. And similarly, we
obtain that $(W^{-c}[c],j_!T[d])\cong
(W^\bullet,j_!T[d])$. Owning to the previous paragraph, we get
$(Z^\bullet,j_!T[d])\cong (W^\bullet,j_!T[d])$, for any $T\in \mod C$. That is,
$\Hom_{\mathcal{D}A}(j_!M,j_!(-)[d]) \cong\Hom_{\mathcal{D}A}(j_!N,j_!(-)[d])| _{\mod C}.$

On the other hand, the equation (1) implies that
$\Hom_{\mathcal{D}A}(j_!M,j_!(-)[d]) \ncong\Hom_{\mathcal{D}A}(j_!N,j_!(-)[d])| _{\mod C},$
which is a contradiction.

\medskip

Above all, we obtain $$\left\{\begin{array}{ll}\Ext_A^{d+a-c}(Z^{-c},-)\ncong\Ext_A^{d+a-c}(W^{-c},-),& (3)
\vspace{2mm}  \\ \Ext_A^t(Z^{-c},-)\cong\Ext_A^t(W^{-c},-)\  \mbox{for}\   t\geq d-c+b+1\ \ &(4)\end{array}\right.$$
In particular, we get from equation (3) that $Z^{-c}\ncong W^{-c}$, and then,
dropping the common direct summands if necessary, we may assume $\add (Z^{-c})\cap \add (W^{-c})=0$.
Therefore, $(Z^{-c}, W^{-c})$ is a $m$-Division
of $Z^{-c}\oplus W^{-c}$, where $m \geq d+a-c$.
As a consequence, for any $L\in \mod C$
with $\phi _C(L)=d>n$, there exists $Z^{-c}\oplus W^{-c}\in \mod A$
with $\phi _A (Z^{-c}\oplus W^{-c})\geq m \geq d-n$.
Thus, if $\phi \dim(C) =\infty$ then so is $\phi  \dim(A)$ and the inequality
(a) holds true. If $\phi \dim(C)< \infty$, we may take $L\in \mod C$
with $\phi _C(L)=d=\phi \dim(C)>n$, then
$\phi  \dim(A)\geq \phi (Z^{-c}\oplus W^{-c})\geq d-n$,
that is, $\phi  \dim(A)\geq \phi  \dim(C)-w(_CX)$.

\medskip

(b) This can be proven in a similar way as we did in (a), just considering
the functors $i_*$ and $i^!$ instead of $j_!$ and $j^*$. Note that
the condition $i_*B \in K^{b}(\proj A)$ ensure that $i^!$ can be restricted
to $D^b(\mod)$.

\medskip

(c) If one of $\pd (_AY), \pd (Y^*_A)$ and $w(X_A)$ is infinite, the statement holds obviously.
Now assume all of them are finite.
If $\phi  \dim(A)\leq \pd (_AY)+\pd (Y^*_A)$
or $\phi  \dim(A)\leq w(X_A)$, (c) also holds true.
Otherwise, there exists $L\in \mod A$ with $\phi (L)=d$, where $d>\pd (_AY)+\pd (Y^*_A)$ and $d>w(X_A)$.
It follows from Theorem~\ref{theorem-phi} that there exist $M, N\in \add L$
such that
$$\left\{\begin{array}{ll}\Ext_A^d(M,-)\ncong \Ext_A^d(N,-),& (5)
\vspace{2mm}  \\ \Ext_A^t(M,-)\cong \Ext_A^t(N,-)\  \mbox{for} \ t\geq d+1.\ \ &(6)\end{array}\right.$$
Let's prove (c) by the following four steps.

\medskip

{\it Step 1.}
Assume $\sup(_AY)=t$ and $\pd (_AY)=e$. Then $i^*M\cong M\otimes _A^LY \in \mathcal{D}(B)_{[-e,-t]}$
and $i^*N \in \mathcal{D}(B)_{[-e,-t]}$.
It follows from Lemma~\ref{lemma-projective-truncation} that $i^*M$ and $i^*N$
can be replaced, respectively, by $M^\bullet , N^\bullet \in \mathcal{D} (B) _{[-e,
-t]}$ with $M^i, N^i\in \proj B$ for $i\in [-e+1,
-t].$
Now we claim $$\Ext_B^t(M^{-e},-)\cong\Ext_B^t(N^{-e},-), \forall
\ t\geq d-e+\pd (Y_B)+1.$$

Indeed, this can be proved analogy to that of part (a): Firstly,
we claim $\Hom_{\mathcal{D}A}(M,i_*(-)[t]) \cong\Hom_{\mathcal{D}A}(M,i_*(-)[t]) | _{\mod B}$
for any $t\geq d+\pd (Y_B)+1.$ Then by adjointness,
$\Hom_{\mathcal{D}B}(i^*M,-[t]) \cong\Hom_{\mathcal{D}B}(i^*N,-[t]) | _{\mod B}$
for any $t\geq d+ \pd (Y_B)+1.$ Finally, applying suitable functor
to the corresponding triangle, we get $\Ext_B^t(M^{-e},-)\cong\Ext_B^t(N^{-e},-)$
for any $t\geq d-e+\pd (Y_B)+1$.

\medskip

{\it Step 2.} Assume $\sup(X_A)=a$ and $\pd(X_A)=b$. Then $\sup(_AX^*)=-b$, $\pd(_AX^*)=-a$
and thus, $j^*M\cong M\otimes _A^LX^*\in \mathcal{D}(C)_{[a,b]}$, $j^*N \in \mathcal{D}(C)_{[a,b]}$.
It follows from Lemma~\ref{lemma-projective-truncation} that $j^*M$ and $j^*N$
can be replaced, respectively, by $U^\bullet , V^\bullet \in \mathcal{D} (C) _{[a,
b]}$ with $U^i, V^i\in \proj A$ for $i\in [a+1,
b].$
Using the same method as Step 1, we get
$$\Ext_C^t(U^{a},-)\cong\Ext_C^t(V^{a},-), \forall
\ t\geq d+a+\pd (X_C^*)+1.$$

\medskip

{\it Step 3.} Assume $\pd (Y_A^*)=f$. We claim $\Ext_B^{d-e-f}(M^{-e},-)\ncong\Ext_B^{d-e-f}(N^{-e},-)$
or $\Ext_C^{d+a-b}(U^{a},-)\ncong\Ext_C^{d+a-b}(V^{a},-)$.

\medskip

Otherwise, $\Ext_B^{d-e-f}(M^{-e},-)\cong\Ext_B^{d-e-f}(N^{-e},-)$ and $\Ext_C^{d+a-b}(U^{a},-)\cong\Ext_C^{d+a-b}(V^{a},-)$.
For any $S\in \mod B$,
applying the functor $\Hom_{\mathcal{D}B}(-,S[d-f])$ to the triangle
$$\tau _{\geq -e+1}M^\bullet \longrightarrow M^\bullet \longrightarrow M^{-e}[e] \longrightarrow ,$$
we get an exact sequence
$$(\tau _{\geq -e+1}M^\bullet [1], S[d-f])\longrightarrow
(M^{-e}[e],S[d-f]) $$ $$\longrightarrow
(M^\bullet,S[d-f])\longrightarrow(\tau _{\geq -e+1}M^\bullet,S[d-f]).$$
Since $d-f>e+f-f=e$, by \cite[Lemma 1.6]{Kato98}, we get
$(\tau _{\geq -e+1}M^\bullet,S[d-f])=0=(\tau _{\geq -e+1}M^\bullet [1], S[d-f])$. Therefore,
$(M^{-e}[e],S[d-f]) \cong
(M^\bullet,S[d-f])$. And similarly, we
obtain that
 $(N^{-e}[e],S[d-f]) \cong
(N^\bullet,S[d-f])$. Therefore, $(M^\bullet,S[d-f])\cong (N^\bullet,S[d-f])$
and thus $(i^*M,S[d-f])\cong (i^*N,S[d-f])$, for any
$S \in \mod B$.

Assume $\sup (Y_A^*)=s$. For any $T \in \mod A$,
$H^n(i^!T)= \Hom _{\mathcal{D} B}(B, i^!T[n])=\Hom _{\mathcal{D} A}(i_*B, T[n])=0$,
for any $n>f$ or $n<s$. Therefore,
$i^!T \in \mathcal{D}(B)_{[s,f]}$. It follows from Lemma~\ref{lemma-projective-truncation} that the complexes $i^!T$
can be replaced by $T^\bullet \in \mathcal{D} (B) _{[s,
f]}$ with $T^i \in \inj B$ for $i\in [s,
f-1].$
Applying $\Hom_{\mathcal{D}B}( i^*M, -[d])$ to the triangle
$$T^{f}[-f] \longrightarrow T^\bullet \longrightarrow \tau _{\leq f-1}T^\bullet \longrightarrow ,$$
we get an exact sequence
$$(i^*M,\tau _{\leq f-1}T^\bullet[d-1])\longrightarrow
(i^*M,T^{f}[d-f])\longrightarrow
(i^*M,T^\bullet[d])\longrightarrow(i^*M,\tau _{\leq f-1}T^\bullet[d]).$$
Since $d>f+e$,
it follows from \cite[Lemma 1.6]{Kato98} that $(i^*M,\tau _{\leq f-1}T^\bullet[d])=0
=(i^*M,\tau _{\leq f-1}T^\bullet[d-1])$. Therefore, $(i^*M,T^{f}[d-f])\cong
(i^*M,T^\bullet[d])$. And similarly, we
obtain that $(i^*N,T^{f}[d-f])\cong
(i^*N,T^\bullet[d])$. By the statement of the previous paragraph, we get
$(i^*M,T^\bullet[d])\cong (i^*N,T^\bullet[d])$. That is,
$(i^*M,i^!T[d])\cong (i^*N,i^!T[d])$, for any $T\in \mod A$.

\medskip

Similarly, $\Ext_C^{d+a-b}(U^{a},-)\cong\Ext_C^{d+a-b}(V^{a},-)$ implies that
$(j^*M,j^*T[d])\cong (j^*N,j^*T[d])$, for any $T\in \mod A$.

\medskip

Applying the functor $\Hom_{\mathcal{D}A}( -, T[d])$ to the triangles
$$j_!j^*M \longrightarrow M \longrightarrow i_*i^*M \longrightarrow \ \ \mbox{and}
\ \ j_!j^*N \longrightarrow N \longrightarrow i_*i^*N \longrightarrow $$
and using adjoint property,
we get two exact sequences
$$(i^*M,i^!T[d])\longrightarrow
(M,T[d])\longrightarrow
(j^*M,j^*T[d])$$
$$(i^*N,i^!T[d])\longrightarrow
(N,T[d])\longrightarrow
(j^*N,j^*T[d]).$$
Since $(i^*M,i^!T[d])=(i^*N,i^!T[d])$ and $(j^*M,j^*T[d])=(j^*N,j^*T[d])$, we get $(M,T[d])\cong (N,T[d])$, for any $T\in \mod A$.
However, this contradicts to the equation (5).

\medskip

Above all, we obtain $$\left\{\begin{array}{ll}\Ext_B^{d-e-f}(M^{-e},-)\ncong\Ext_B^{d-e-f}(N^{-e},-),&\ \ (7)
\vspace{2mm}  \\ \Ext_B^t(M^{-e},-)\cong\Ext_B^t(N^{-e},-)\ \mbox{for} \ t\geq d-e+\pd (Y_B)+1. &\ \ (8)\end{array}\right.$$
or
$$\left\{\begin{array}{ll}\Ext_C^{d+a-b}(U^{a},-)\ncong\Ext_C^{d+a-b}(V^{a},-),& \ \ \ \ \ \ (9)
\vspace{2mm}  \\ \Ext_C^t(U^{a},-)\cong\Ext_C^t(V^{a},-)\ \mbox{for} \ t\geq d+a+\pd (X_C^*)+1. &\ \ \ \ \ \ (10)\end{array}\right.$$

 Therefore, $\phi _B (M^{-e}\oplus N^{-e})\geq d-e-f$
 or $\phi _C (U^{a}\oplus V^{a})\geq d+a-b$. Hence,
$\phi  \dim(B) \geq \phi  \dim(A)-e-f$ or
$\phi  \dim(C)\geq \phi  \dim(A)+a-b$. Therefore,
$\phi  \dim(A)\leq \max \{ \phi  \dim(B) +\pd (_AY)+\pd (Y^*_A), \phi  \dim(C)+w(X_A) \}.$
\end{proof}

Applying Theorem~\ref{theorem4} to trivial recollement, we reobtain
the main result in \cite{FLM15}.
\begin{corollary}{\rm (\cite[Theorem 4.10]{FLM15})}
Let $A$ and $B$ be two derived equivalence algebras and $_BP_A^\bullet$
the corresponding two-side tilting complex. Then
$$\phi  \dim(A)-w(P_A) \leq \phi  \dim(B)\leq \phi  \dim(A)+w(P_A) .$$
\end{corollary}

\noindent {\footnotesize {\bf ACKNOWLEDGMENT.}
This work is supported by Yunnan Applied Basic Research Project 2016FD007.}


\begin{thebibliography}{99}

\bibitem{AKLY13} L. Angeleri H\"{u}gel, S. K\"{o}nig, Q. Liu and D.
Yang, Ladders and simplicity of derived module categories, arXiv:1310.3479v2 [math.RT].

\bibitem{BBD82} A. A. Beilinson, J. Bernstein, P. Deligne, Faisceaux pervers, Ast\'erisque, vol. 100, Soc. Math. France, 1982.

\bibitem{CX12} H. X. Chen and C. C. Xi, Recollements of derived categories II: Algebraic K-theory, arXiv:1212.1879 [math.KT].

\bibitem{CX13} H. X. Chen and C. C. Xi, Recollements of derived categories III: Finitistic dimensions, arXiv:1405.5090v1 [math.RA].

\bibitem{Chen09} X. W. Chen, Singularity categories, Schur functors and triangular matrix rings, Algebr. Represent. Theor. 12 (2009), 181-191.

\bibitem{CPS88} E. Cline, B. Parshall and L. Scott, Finite dimensional algebras and highest weight categories, J. Reine Angew. Math. 391 (1988), 85-99.

\bibitem{FLM15} S. Fernandes, M. Lanzilotta and O. Mendoza, The $\phi$-dimension: A new homological measure, Algebr. Represent. Theor. 18 (2015), 463-476.

\bibitem{GR92} P. Gabriel and A.V. Roiter, Representations of finite-dimensional algebras, Encyclopaedia Math. Sci.,  vol. 73, Springer-Verlag, 1992.

\bibitem{Han14} Y. Han, Recollements and Hochschild theory, J. Algebra 397 (2014), 535-547.

\bibitem{Hap88} D. Happel, Triangulated categories in the representation theory of finite dimensional algebras, London Math. Soc. Lectue Notes Ser. 119, 1988.

\bibitem{Hap93} D. Happel, Reduction techniques for homological conjectures, Tsukuba J. Math. 17 (1993), 115-130.

\bibitem{HL13} F. Huard and M. Lanzilotta, Self-injective right artinian rings and Igusa-Todorov functions, Algebr. Represent. Theor. 16(3) (2013), 765-770.

\bibitem{HLM08} F. Huard, M. Lanzilotta and O. Mendoza, An approach to the finitistic dimension conjecture, J. Algebra 319 (2008), 3918-3934.

\bibitem{IT05} K. Igusa and G. Todorov, On the finitistic global dimension conjecture for artin algebras, Representation of algebras and related topics, 201-204. Field Inst. Commun., 45, Amer. Math. Soc., Providence, RI, 2005.

\bibitem{Kato98} Y. Kato, On derived equivalent coherent rings, Comm. Algebra 30 (2002), 4437-4454.


\bibitem{Kel98} B. Keller, Invariance and localization for cyclic homology of DG algebras, J. Pure Appl. Algebra 123 (1998), no. 1-3, 223-273.

\bibitem{Koe91} S. K\"{o}nig, Tilting complexes, perpendicular categories and recollements of derived categories of rings, J. Pure Appl. Algebra 73 (1991), 211-232.

\bibitem{LM16} M. Lanzilotta and G. Mata, Igusa-Todorov functions for Artin algebras, arXiv:1605.09766v1 [math.RT].

\bibitem{Pan13} S. Y. Pan, Recollements and Gorenstein Algebras, Int. J. Algebra, Vol. 7 (2013), 829-832.

\bibitem{PX09} S. Y. Pan and C. C. Xi, Finiteness of finitistic dimension is invariant of derived equivalences, J. Algebra 322 (2009), 21-24.

\bibitem{QH16} Y.Y. Qin and Y. Han, Reducing homological
conjectures by $n$-recollements, Algebr. Represent. Theor. 19 (2016), no. 2, 377-395.

\bibitem{Wei09} J. Q. Wei, Finitistic dimension and Igusa-Todorov algebras, Advances in Math. 222 (2009), 2215--2226.

\bibitem{Wie91} A. Wiedemann, On stratifications of derived module categories, Canad. Math. Bull. 34 (1991), 275-280.

\bibitem{XZ12} B. L. Xiong, P. Zhang, Gorenstein-projective modules over triangular matrix Artin algebras, J. Algebra Appl. 11 (2012), 1802-1812.

\bibitem{Xu13} D. M. Xu, Generalized Igusa-Todorov function and finitistic dimensions, Arch. Math. 100 (2013), 309--322.

\bibitem{Zak69} A. Zaks, Injective dimensions of semiprimary rings, J. Algebra 13 (1969), 73-89.
\end{thebibliography}
\end{document}